\documentclass[11pt]{amsart}
\usepackage[utf8]{inputenc}
\usepackage{amsfonts, amssymb, amsmath, amsthm, color, float, enumerate, bigints}
\usepackage{url}
\usepackage{mathrsfs}
\usepackage[unicode,psdextra]{hyperref}
\usepackage{pgfplots,tikz}
\pgfplotsset{compat=1.18}

\title[]{Weighted mixed inequalities for commutators of Schrödinger type operators}
\author{}
\date{}

\usepackage[a4paper, left=2cm, right=2cm, top=3cm, bottom=3cm]{geometry} 

\theoremstyle{plain}
   \newtheorem{teo}{Theorem}
   
   \newtheorem{lema}[teo]{Lemma}
   \newtheorem{propo}[teo]{Proposition}
   
\theoremstyle{definition}
   
\theoremstyle{remark}
 \newtheorem{obs}{Remark}
 
 \newtheorem{afirmacion}{Claim}

\numberwithin{equation}{section}
\numberwithin{teo}{section}
\allowdisplaybreaks

\definecolor{aquamarine}{rgb}{0.5, 1.0, 0.83}
\definecolor{americanrose}{rgb}{1.0, 0.01, 0.24}
\definecolor{arsenic}{rgb}{0.23, 0.27, 0.29}
\definecolor{blizzardblue}{rgb}{0.67, 0.9, 0.93}
\definecolor{blush}{rgb}{0.87, 0.36, 0.51}
\definecolor{celestialblue}{rgb}{0.29, 0.59, 0.82}
\definecolor{chocolate(web)}{rgb}{0.82, 0.41, 0.12}
\definecolor{brightpink}{rgb}{1,0,0.5}
\definecolor{cadmiunred}{rgb}{0.89,0,0.13}

\hypersetup{
	colorlinks = true,
	linkcolor = cadmiunred,
	anchorcolor = blue,
	citecolor = brightpink,
	filecolor = blue,
	urlcolor = blue
}

\newcounter{BPR}

\begin{document}

	\author[F. Berra]{Fabio Berra}
	\address{Fabio Berra, CONICET and Departamento de Matem\'{a}tica (FIQ-UNL),  Santa Fe, Argentina.}
	\email{fberra@santafe-conicet.gov.ar}
	
	\author[G. Pradolini]{Gladis Pradolini}
	\address{Gladis Pradolini, CONICET and Departamento de Matem\'{a}tica (FIQ-UNL),  Santa Fe, Argentina.}
	\email{gladis.pradolini@gmail.com}
%
	\author[J. Recchi]{Jorgelina Recchi}
	\address{Jorgelina Recchi, Departamento de Matemática, Universidad Nacional del Sur (UNS), Instituto de Matematica (INMABB), Universidad Nacional del Sur-CONICET,  Bahía Blanca, Argentina.}
	\email{jrecchi@gmail.com}
	
	\thanks{The author were supported by CONICET, UNL, ANPCyT and UNS}
	
	\subjclass[2010]{35J10, 42B20}
	
	\keywords{Schrödinger operators, critical radious function, commutators, weights}

\begin{abstract}	
We obtain weighted mixed inequalities for the first order commutator of singular integral operators in the Schrödinger setting. Concretely, for $0<\delta\leq 1$ we give estimates of commutators of Schrödinger-Calderón-Zygmund operators of $(s,\delta)$ type with $1<s\leq \infty$, and $\text{BMO}(\rho)$ symbols associated to a critical radious function $\rho$.

Our results generalizes some previous estimates about mixed inequalities for Schrödinger type operators. We also deal with $A_p^\rho$ weights,  which can be understood as a perturbation of the $A_p$ Muckenhoupt classes by means of function $\rho$.
\end{abstract}

\maketitle


\section{Introduction}

One of the most relevant problems in Harmonic Analysis is to determine the continuity properties of certain maximal operators that control, in some sense, a wide variety of operators of integral type. For example, the Hardy-Littlewood maximal function, $M$, moderates the behavior of Calderón-Zygmund operators by controlling the norm of the latter with different measures. See for example \cite{Coifman}, \cite{Coifman-Fefferman-74}.

It is well known that $M$ is bounded in $L^{p}(w)$ if $w\in A_p$, the Muckenhoupt classes of weights, which is defined by
\begin{equation*}
\sup_{Q\subset \mathbb{R}^n}\left(\frac{1}{|Q|}\int_Q w\right)^{1/p}\left(\frac{1}{|Q|}\int_Q w^{1-p'}\right)^{1/p'}<\infty,
\end{equation*} 
where the supremum is taken over the cubes in $\mathbb{R}^n$ with sides parallel to the coordinate axes. The proof of this result involves an interpolation argument combined with the endpoint behavior of $M$ when $p=1$ (\cite{Muck72}).

Motivated in giving an alternative proof of the continuity properties of $M$ in $L^p(w)$, in \cite{Sawyer} E. Sawyer proved an inequality on the real line involving $A_1$ weights and certain operator, which is a perturbation of $M$. Concretely, this inequality establishes that, if $u$, $v\in A_1$  then the following estimate holds
\begin{equation}\label{mixta Sawyer}
uv\left(\left\{x\in\mathbb{R}: \frac{M(fv)(x)}{v(x)}>\lambda\right\}\right)\leq \frac{C}{\lambda}\int_{\mathbb{R}}|f(x)|\,u(x)v(x)\,dx.
\end{equation}
We shall refer to this type of estimates as weighted mixed inequalities and they can be seen as a generalization of the weak $(1,1)$ type of $M$, which corresponds to $v=1$. It can also be identified as the weak $(1,1)$ type of the operator $ S_v(f)=M(fv)/v$ with measure $uv$. In the same article, the author also conjectured that a similar estimate should still hold for the Hilbert transform, $H$. 

In \cite{CruzUribe-Martell-Perez} the authors proved Sawyer's conjecture not only for $H$ but also for Calderón-Zygmund operators (CZO) and the Hardy-Littlewood maximal function in $\mathbb{R}^n$. They also included similar inequalities for $u\in A_1$ and $v\in A_\infty(u)$. Under these assumptions we have that the product $uv\in A_\infty$ and so the Calderón-Zygmund decomposition can be performed in order to obtain the corresponding result for $M_{\mathcal{D}}$, the dyadic Hardy-Littlewood maximal operator. Then, an extrapolation result gave the corresponding estimate for $M$ and CZO.

Inspired in the techniques developed in \cite{PP01}, in \cite{BCP-M} the authors gave an alternative proof of the results in \cite{CruzUribe-Martell-Perez} for CZO and their commutators with BMO symbols, for the case $u\in A_1$ and $v\in A_\infty(u)$. These hypotheses allowed to apply the Calderón-Zygmund decomposition with respect to the measure $v(x)\, dx$ and to obtain a direct proof of the mentioned result. In this direction, mixed type inequalities were also explored in, for example, \cite{B19}, \cite{B24}, \cite{BCP-JMAA}, \cite{BCP-MN}, \cite{BCP-Coll},    \cite{LOP}.

On the Schrödinger setting, in \cite{shen} the author studied $L^p$--continuity of the operator $T=-\Delta +V$, where $V$ is a nonnegative potential satisfying certain reverse Hölder condition. Later on, in \cite{BHS11JMAA} the authors proved $L^p(w)$ boundedness properties for the Riesz-Schrödinger transform, introducing the classes of weights associated in this setting, which generalize the classical Muckenhoupt $A_p$ weights, preserving many properties of them in this context,  but with subtle modifications. The development of this theory continued in \cite{BCH13}, where weighted inequalities for a wide family of singular integral operators, which are extension of classical CZO, were established.      

On the other hand, weighted weak type estimates for Schrödinger singular integral operators were studied in \cite{BCH16}. Moreover, in \cite{BHS11} a modular weak type inequality involving the Young function $\varphi(t)=t(1+\log^+ t)$ for the commutator of the Riesz transform was obtained. 

A first approach in obtaining weighted mixed inequalities in the Schrödinger setting was developed in \cite{BPQ24} where the authors (particularly, the first two of them in the present article), proved estimates in the spirit of \eqref{mixta Sawyer} for singular integral operators with kernels satisfying different size and regularity conditions associated to a critical radius function $\rho$. Concretely, the estimates involve perturbations of Schrödinger-Calderón-Zygmund of $(\infty,\delta)$ and $(s,\delta)$ type (see the definition below). The conditions on the weights appearing in these inequalities are of the type $u\in A_1^\rho$ and $v\in A_\infty^\rho(u)$ for the former with a slight modification for the latter. More general mixed inequalities for these operators and the Hardy-Littlewood maximal operator in this new setting was recently settled in \cite{BPQ24(2)}.

In this article we consider commutators of Schrödinger-Calderón-Zygmund of $(\infty,\delta)$ and $(s,\delta)$ type with symbols belonging to BMO classes related to a critical radius function $\rho$. Our techniques rely upon a generalized Calderón-Zygmund decomposition in this framework respect to a certain measure, combined with the decay properties of the function $\rho$, extending those obtained in \cite{BCP-M}. A generalization of many well known properties that hold in the classical context is required in this new environment, and non trivial adaptations are necessary in order to obtain the main results.

In order to state our main results we give some preliminaries and definitions.

Throughout the article we shall be dealing with a \emph{critical radious function} $\rho\colon \mathbb{R}^n\to (0,\infty)$, that is, a function $\rho$ for which there exist two constants $C_0,N_0\geq 1$ such that the inequality
\begin{equation} \label{eq: constantesRho}
	C_0^{-1}\rho(x) \left(1+ \frac{|x-y|}{\rho(x)}\right)^{-N_0}
	\leq \rho(y)
	\leq C_0 \,\rho(x) \left(1+ \frac{|x-y|}{\rho(x)}\right)^{\tfrac{N_0}{N_0+1}}
	\end{equation}
	holds for every $x,y\in\mathbb{R}^n$.

 Given $0<\delta\leq1$ we say that a linear operator $T$ is a  \emph{Schr\"odinger-Calder\'on-Zygmund operator of $(\infty,\delta)$ type }($(\infty,\delta)-SCZO$) if
\begin{enumerate}[\rm(I)]
	\item $T$ is bounded from $L^1$ into $L^{1,\infty}$;
	\item $T$ has an associated kernel $K\colon\mathbb{R}^n\times\mathbb{R}^n\rightarrow\mathbb{R}$, such that the representation
	\begin{equation}\label{eq: representacion integral de T}
	Tf(x)=\int_{\mathbb{R}^n} K(x,y)f(y)\,dy
	\end{equation}
 holds for $f\in L_c^{\infty}$ and almost every $x\notin \text{supp}(f)$;
	\item
	for every $N>0$ there exists a constant $C_N$ such that
	\begin{equation}\label{eq: condicion de tamaño (inf,delta)}
	|K(x,y)| \leq
	\frac{C_N}{|x-y|^{n}} \left(1+ \frac{|x-y|}{\rho(x)}\right)^{-N},\,\,\, x\neq y, 
	\end{equation}		
	and there exists $C$ such that
	\begin{equation}\label{eq: condicion de suavidad (inf,delta)}
	|K(x,y)-K(x,y_0)|
	\leq C \frac{|y-y_0|^{\delta}}{|x-y|^{n+\delta}},\,\,\,\text{when}\,\,
	|x-y|>2|y-y_0|.
	\end{equation}
\end{enumerate}

We shall also consider a wider class of operators with kernels satisfying another type of regularity. For $1<s<\infty$ and $0<\delta\leq1$, we shall say that a linear operator $T$ is a  \emph{Schr\"odinger-Calder\'on-Zygmund operator of $(s,\delta)$ type} ($(s,\delta)-SCZO$) if
\begin{enumerate}
	\item[($\rm{I}_s$)]\label{tipodebil} $T$ is bounded on $L^{p}$ for $1<p<s$. 
	\item[($\rm{II}_s$)] $T$ has an associated kernel $K:\mathbb{R}^n\times\mathbb{R}^n\rightarrow\mathbb{R}$, in the sense that
	\begin{equation*}
	Tf(x)=\int_{\mathbb{R}^n} K(x,y)f(y)\,dy,\,\,\,\,
	f\in L_c^{s'} \,\,\text{and}\,\,x\notin \text{supp}f.
	\end{equation*}
\item[($\rm{III}_s$)]	Further, 
	for each $N>0$ there exists a positive constant $C_N$ such that
	\begin{equation}\label{TamHorm}
	\left( \int_{R<|x_0-x|<2R} |K(x,y)|^{s}\,dx \right)^{1/s}\\ \leq C_N R^{-n/s'} \left(1+\frac{R}{\rho(x_0)}\right)^{-N},
	\end{equation}
	for  $|y-x_0|<R/2$, and
	there exists a positive constant $C$ such that
	\begin{equation}\label{suav-horm}
	\left( \int_{R<|x-y_0|<2R} |K(x,y)-K(x,y_0)|^{s}\,dx\right)^{1/s} \\
	\leq C R^{-n/s'} \left(\frac{r}{R}\right)^{\delta},
	\end{equation}
	for $|y-y_0|<r\leq \rho (y_0)$,  $r<R/2$.
\end{enumerate}

\medskip 

Given $b\in L^1_{\text{loc}}$ and a linear operator $T$, the \emph{first order commutator} $T_b$ of $T$ is formally defined by
\[T_bf=[b,T]f=bTf-T(bf).\]

Given $\theta\geq 0$ and a critical radius function $\rho$, the \emph{bounded mean oscillation space}  $\text{BMO}_\theta(\rho)$ is defined as the collection of locally integrable functions $f$ for which there exists $C>0$ such that the inequality
\[\frac{1}{|Q|}\int_Q |f-f_Q|\leq C\left(1+\frac{r}{\rho(x)}\right)^{\theta}\]
holds for every cube $Q=Q(x,r)$. The smallest constant $C$ satisfying the inequality above is denoted by $\|f\|_{\text{BMO}_\theta(\rho)}$. The space $\text{BMO}(\rho)$ shall be understood as the collection of the $\text{BMO}_\theta(\rho)$ spaces for $\theta\geq 0$, that is, 
\[\text{BMO}(\rho)=\bigcup_{\theta\geq 0}\text{BMO}_\theta(\rho).\]

\medskip

We are now in a position to state our main results. For the definitions of the classes of weights see the next section. Throughout the article we denote $\Phi(t)=t(1+\log^+t)$.

\begin{teo}\label{teo: mixta para SCZO inf,delta}
Let $u\in A_1^{\rho}$ and $v\in A_\infty^\rho(u)$. Given $0<\delta\leq 1$, $T$ a $(\infty,\delta)-SCZO$ and $b\in \text{BMO}(\rho)$, the inequality
\[uv\left(\left\{x\in\mathbb{R}^n: \frac{|T_b(fv)(x)|}{v(x)}>\lambda\right\}\right)\leq C\int_{\mathbb{R}^n}\Phi\left(\frac{|f(x)|}{\lambda}\right)u(x)v(x)\,dx\]
holds for every positive $\lambda$.
\end{teo}

\begin{teo}\label{teo: mixta para SCZO s,delta}
Let $1<s<\infty$, $0<\delta\leq 1$, and $u,v$ weights such that  $u^{s'}\in A_1^{\rho}$ and $v\in A_\infty^\rho(u^\beta)$, for some $\beta>s'$. If $T$ is a $(s,\delta)-SCZO$ and $b\in \text{BMO}(\rho)$ then the inequality
\[uv\left(\left\{x\in\mathbb{R}^n: \frac{|T_b(fv)(x)|}{v(x)}>\lambda\right\}\right)\leq C\int_{\mathbb{R}^n}\Phi\left(\frac{|f(x)|}{\lambda}\right)u(x)v(x)\,dx\]
holds for every positive $\lambda$.
\end{teo}

\section{Preliminaries and definitions}

We shall establish our estimates on the Euclidean space $\mathbb{R}^n$. By   $Q(x,r)$ we understand a cube with sides parallel to the coordinate axes, with centre at the point $x$ and radius $r$. We shall also denote with $\ell(Q)$ the length of the edges of $Q$ if necessary. Notice that we have the relation $r=\sqrt{n}\,\ell(Q)/2$. We shall also distinguish the cubes in terms of the critical radius function: if $Q=Q(x,r)$ verifies $r\leq \rho(x)$ we say that $Q$ is a \emph{subcritical cube} and we call it \emph{critical} when $r=\rho(x)$.

We now establish the classes of weights associated to $\rho$, which can be understood as a perturbation of the  weighted Muckenhoupt classes by means of $\rho$. Given a weight $u$, $1<p<\infty$ and a nonnegative number $\theta$, we say that $w\in A_p^{\rho,\theta}(u)$ if there exists $C>0$ such that the inequality
\begin{equation}\label{eq: clase Ap,rho,theta(u)}
\left(\frac{1}{u(Q)}\int_Q wu\right)^{1/p}\left(\frac{1}{u(Q)}\int_Q w^{1-p'}u\right)^{1/p'}\leq C\left(1+\frac{r}{\rho(x)}\right)^{\theta}
\end{equation} 
holds for every cube $Q=Q(x,r)$. The expression $u(Q)$ stands for  $\int_Qu$, as usual.

Similarly, $w\in A_1^{\rho,\theta}(u)$ if there exists a positive constant $C$ such that
\begin{equation}\label{eq: clase A1,rho,theta(u)}
\frac{1}{u(Q)}\int_Q wu\leq C\left(1+\frac{r}{\rho(x)}\right)^{\theta}\inf_Q w,
\end{equation} 
for every cube $Q=Q(x,r)$.  We also define $A_\infty^{\rho,\theta}(u)=\bigcup_{p\geq 1} A_p^{\rho,\theta}(u)$. The smallest constants $C$ in $(\ref{eq: clase Ap,rho,theta(u)})$ and $(\ref{eq: clase A1,rho,theta(u)})$ will be denoted by $[w]_{A_p^{\rho,\theta}(u)}$. When $u=1$ we shall denote it by $[w]_{A_p^{\rho,\theta}}$.

For $1\leq p\leq \infty$, the $A_p^\rho(u)$ class is the collection of all the $A_p^{\rho,\theta}(u)$ classes for $\theta\geq 0$, that is
\[A_p^\rho(u)=\bigcup_{\theta\geq 0} A_p^{\rho,\theta}(u).\]
When $u=1$ we shall directly write $A_p^{\rho,\theta}(u)=A_p^{\rho,\theta}$ and $A_p^\rho(u)=A_p^\rho$, which were introduced in~\cite{BHS11JMAA}.  

\begin{obs}\label{obs: v en Ap implica v^{1-p'} en Ap'} It is known that $v\in A_p^{\rho}$ implies $v^{1-p'}\in A_{p'}^{\rho}$, see \cite{BHS11JMAA}.

\end{obs}

The following lemma establishes a  property that relates  $A_p^{\rho}(u)$ and $A_p^\rho$ classes. 

\begin{lema}\label{lema: propiedades de Ap,rho}
The following properties hold.
\begin{enumerate}[\rm (a)]
    \item\label{item: lema: propiedades de Ap,rho - item a}  If $1\leq p\leq \infty$, $u\in A_1^{\rho}$ and $v\in A_p^\rho(u)$, then $uv\in A_p^\rho$.
    \item \label{item: lema: propiedades de Ap,rho - item b} If $0<\alpha<1$, $1<p\leq \infty$, $u\in A_1^\rho$ and $v\in A_p^\rho(u)$, then $v\in A_p^{\rho}(u^{\alpha})$.
    \item \label{item: lema: propiedades de Ap,rho - item c} If $u\in A_1^\rho$ then $A_p^\rho(u)\subseteq A_p^\rho$.
\end{enumerate}
\end{lema}

\begin{proof}
    We only give the proof of item~\eqref{item: lema: propiedades de Ap,rho - item c}, since the proofs of \eqref{item: lema: propiedades de Ap,rho - item a} and \eqref{item: lema: propiedades de Ap,rho - item b} can be found in \cite{BPQ24}. Let $1<p<\infty$ and $w\in A_p^\rho(u)$. By hypothesis, there exist nonnegative numbers $\theta_1$ and $\theta_2$ such that $u\in A_1^{\rho,\theta_1}$ and $w\in A_p^{\rho,\theta_2}(u)$. For any cube $Q=Q(x,r)$ we have
    \begin{align*}
        \left(\frac{1}{|Q|}\int_Q w\right)^{1/p}\left(\frac{1}{|Q|}\int_Q w^{1-p'}\right)^{1/p'}&\leq \left(\sup_Q u^{-1}\right)\frac{u(Q)}{|Q|}\left(\frac{1}{u(Q)}\int_Q wu\right)^{1/p}\left(\frac{1}{u(Q)}\int_Q w^{1-p'}u\right)^{1/p'}\\
        &\leq [u]_{A_1^{\rho,\theta_1}}[w]_{A_p^{\rho,\theta_2}(u)}\left(1+\frac{r}{\rho(x)}\right)^{\theta_1+\theta_2}.
    \end{align*}
    Thus $w\in A_p^{\rho,\theta_1+\theta_2}$ and $[w]_{A_p^{\rho,\theta_1+\theta_2}}\leq [u]_{A_1^{\rho,\theta_1}}[w]_{A_p^{\rho,\theta_2}(u)}$.
\end{proof}




It is known, as in the classical case, that every $A_p^\rho$ weight satisfies a \emph{reverse Hölder} condition (see, for example, \cite{BHS11JMAA}). For $1<s<\infty$ and $\theta\geq 0$ we say that $w\in \text{RH}_s^{\rho,\theta}$ if there exists $C>0$ such that the inequality
\[\left(\frac{1}{|Q|}\int_Q w^s\right)^{1/s}\leq C\left(\frac{1}{|Q|}\int_Q w\right)\left(1+\frac{r}{\rho(x)}\right)^{\theta}\]
holds for every cube $Q=Q(x,r)$. We also have, as before, $\text{RH}_s^{\rho}=\bigcup_{\theta\geq 0}\text{RH}_s^{\rho,\theta}$.  The smallest constant $C$ in the inequality above will be denoted by $[w]_{\text{RH}_s^{\rho,\theta}}$.

We say that $\varphi\colon [0,\infty)\to[0,\infty)$ is a Young function if it is  a strictly increasing and convex function that verifies $\varphi(0)=0$ and $\varphi(t)\rightarrow\infty$ when $t\rightarrow \infty$. 

The complementary function $\tilde\varphi$ of such $\varphi$ is defined  by
\[\tilde\varphi(t)=\sup\{ts-\varphi(s): s\geq 0\}.\]
If $\varphi$ and $\tilde\varphi$ are Young functions, the relation
\begin{equation}\label{eq: producto de inversas como t}
\varphi^{-1}(t)\tilde\varphi^{-1}(t)\approx t
\end{equation} 
holds for every $t$, where $\varphi^{-1}$ stands for the generalized inverse function (see, for example, \cite{KR} or \cite{raoren}).

Given a Young function $\varphi$ and a weight $w$, the quantity  $\|f\|_{\varphi, Q, w}$ denotes the weighted Luxemburg average of $f$ over $Q$, defined by
\[\|f\|_{\varphi,Q,w}=\inf\left\{\lambda>0: \frac{1}{w(Q)}\int_Q\varphi\left(\frac{|f|}{\lambda}\right)w\leq 1\right\}.\]
In fact, the infimum above is actually a minimum, since it can be seen that
\begin{equation}\label{eq: norma realiza el infimo}
    \frac{1}{w(Q)}\int_Q\varphi\left(\frac{|f|}{\|f\|_{\varphi,Q,w}}\right)w\leq 1.
\end{equation}
The following lemma gives an equivalent expression for the norm $\|\cdot\|_{\varphi,Q,w}$ that will be useful in our estimates. We include the proof for the sake of completeness.

\begin{lema}
    Let $w$ be a weight, $Q$ a cube and $f$ such that $\varphi(|f|)\in L^{1}_{\mathrm{loc}}(w)$. Then we have that
    \begin{equation}\label{eq: norma equivalente}
\|f\|_{\varphi,Q,w}\approx \inf_{\lambda>0}\left\{\lambda +\frac{\lambda}{w(Q)}\int_Q\varphi \left(\frac{|f|}{\lambda}\right)\, w\right\}.
\end{equation}
\end{lema}

\begin{proof}
We first observe that from the convexity of $\varphi$
\begin{align*}
  \frac{1}{w(Q)}\bigintsss_Q \varphi\left(\frac{|f|}{\lambda+\frac{\lambda}{w(Q)}\int_Q \varphi\left(\frac{|f|}{\lambda}\right)w}\right)w&=\frac{1}{w(Q)}\bigintsss_Q \varphi\left(\frac{|f|}{\lambda\left(1+\frac{1}{w(Q)}\int_Q \varphi\left(\frac{|f|}{\lambda}\right)w\right)}\right)w\\
  &\leq \left(1+\frac{1}{w(Q)}\int_Q \varphi\left(\frac{|f|}{\lambda}\right)w\right)^{-1}\frac{1}{w(Q)}\bigintsss_Q \varphi\left(\frac{|f|}{\lambda}\right)w\\
  &\leq 1,
\end{align*}
which implies that 
\[\|f\|_{\varphi,Q,w}\leq \lambda +\frac{\lambda}{w(Q)}\int_Q\varphi \left(\frac{|f|}{\lambda}\right)\, w\]
for every $\lambda>0$. So we obtain that
\[\|f\|_{\varphi,Q,w}\leq \inf_{\lambda>0}\left\{\lambda +\frac{\lambda}{w(Q)}\int_Q\varphi \left(\frac{|f|}{\lambda}\right)\, w\right\}.\]
The other inequality trivially holds if $\|f\|_{\varphi,Q,w}=0$. If this is not the case, notice that
\[\inf_{\lambda>0}\left\{\lambda +\frac{\lambda}{w(Q)}\int_Q\varphi \left(\frac{|f|}{\lambda}\right)\, w\right\}\leq \|f\|_{\varphi,Q,w}+\frac{\|f\|_{\varphi,Q,w}}{w(Q)}\int_Q \varphi\left(\frac{|f|}{\|f\|_{\varphi,Q,w}}\right)w\leq 2\|f\|_{\varphi,Q,w}\]
by virtue of \eqref{eq: norma realiza el infimo}.
\end{proof}

If $\varphi$, $\psi$ and $\eta$ are Young functions that verify
 \[\eta^{-1}(t)\psi^{-1}(t)\lesssim \varphi^{-1}(t)\]
 for every $t\geq t_0>0$, we can conclude that there exists $K_0>0$ such that
 \begin{equation*}
 \varphi(st)\lesssim \eta(s)+\psi(t)
 \end{equation*}
 for $s,t\geq K_0$. As a consequence, the generalized version of Hölder inequality for Luxemburg averages
 \begin{equation}\label{eq: Holder generalizada con promedios Luxemburgo}
 \|fg\|_{\varphi, Q, w}\lesssim \|f\|_{\eta, Q, w}\|g\|_{\psi, Q, w}
 \end{equation}
holds for every weight $w$ and every cube $Q$. Particularly, when we take $w=1$, by means of \eqref{eq: producto de inversas como t} we have that
\begin{equation}\label{eq: desigualdad de Holder generalizada}
\frac{1}{|Q|}\int_Q |fg|\lesssim \|f\|_{\varphi, Q}\,\|g\|_{\tilde\varphi, Q}.
\end{equation}

\section{Auxiliary results}

Let $\mu$ be a measure satisfying the following property: there exist constants $C>0$ and $\sigma\geq 0$ such that the inequality
\begin{equation}\label{eq: propiedad de mu}
   \mu(2Q)\leq C\mu(Q)\left(1+\frac{r}{\rho(x)}\right)^{\sigma} 
\end{equation}
holds for every cube $Q=Q(x,r)$.

\begin{obs}\label{obs: Ap implica duplicante}
It is easy to see that if $w\in A_p^\rho$ for some $1\leq p\le \infty$, then the measure $\mu$ given by $d\mu(x)=w(x)\,dx$ satisfies \eqref{eq: propiedad de mu}.    
\end{obs}

The following result is a slight modification of Lemma 9 in \cite{BHS12}  for doubling measures in the sense of \eqref{eq: propiedad de mu}.

\begin{lema}[Calderón-Zygmund decomposition]\label{lema: CZ decomposition}
Let $\rho$ be a critical radius function satisfying \eqref{eq: constantesRho}, $\theta\geq 0$ and $\mu$ be a measure satisfying \eqref{eq: propiedad de mu}. Then, for any positive $\lambda$ there exists a countable family of disjoint cubes $\{P_j\}$, $P_j=Q(x_j,r_j)$ such that 
\[\left(1+\frac{r_j}{\rho(x_j)}\right)^\theta \lambda <\frac{1}{\mu(P_j)}\int_{P_j}|f|\,d\mu\leq C\left(1+\frac{r_j}{\rho(x_j)}\right)^\sigma \lambda,\]
for some $C>0$ and $\sigma\geq \theta$, depending only on $\rho$ and $\mu$. On the other hand, $|f(x)|\leq \lambda$ for almost every $\displaystyle x\not\in\bigcup_j P_j$.
\end{lema}

The following lemma involves the function $\tilde{\Phi}(t)= {\rm e}^t-1$ which is the conjugate Young function of $\Phi(t)=t(1+\log^+t)$.

\begin{lema}\label{lema: con pesos menor que sin pesos}
Let $w$ be a weight such that $w\in\mathrm{RH}_s^{\rho,\theta}$ for some $s>1$ and $\theta \ge0$. If $Q=Q(x_Q, r_Q)$, then
\[\|f\|_{\tilde{\Phi},Q,w}\leq s'\, 2^{1/s'} [w]_{\mathrm{RH}_s^{\rho,\theta}}\left(1+\frac{r_Q}{\rho(x_Q)}\right)^{\theta}\|f\|_{\tilde{\Phi},Q}.\]
\end{lema}

\begin{proof}
Fix  $\lambda=s'\|f\|_{\tilde{\Phi},Q}$. In order to show that
$\displaystyle \|f\|_{\tilde{\Phi},Q,w}\leq 2^{1/s'}[w]_{\mathrm{RH}_s^{\rho, \theta}} \left(1+\frac{r_Q}{\rho(x_Q)}\right)^{\theta}\lambda$,
 it is enough to prove that
\begin{equation}\label{eq: con pesos menor que sin pesos}
\frac{1}{w(Q)}\int_Q \left(e^{\frac{|f(x)|}{\lambda}}-1\right)w(x)\,dx \leq 2^{1/s'}[w]_{\mathrm{RH}_s^{\rho, \theta}}\left(1+\frac{r_Q}{\rho(x_Q)}\right)^{\theta}.
\end{equation}
Indeed, since $C=2^{1/s'}[w]_{\mathrm{RH}_s^{\rho, \theta}}\left(1+\frac{r_Q}{\rho(x_Q)}\right)^{\theta}>1$, from \eqref{eq: con pesos menor que sin pesos}  we obtain that
\[\frac{1}{w(Q)}\int_Q \left(e^{\frac{|f(x)|}{C\lambda}}-1\right)w(x)\,dx\leq
\frac{1}{Cw(Q)}\int_Q \left(e^{\frac{|f(x)|}{\lambda}}-1\right)w(x)\,dx \leq 1,\]
where we have used that $\psi(\alpha t)\leq \alpha \psi(t)$, for every convex function $\psi$ with $\psi(0)=0$ and every $\alpha\in[0,1]$.
Then we conclude that $\|f\|_{\tilde{\Phi},Q,w}\leq C\lambda$.

Then, let us prove that \eqref{eq: con pesos menor que sin pesos} holds.
From  H\"{o}lder's inequality and the reverse H\"{o}lder condition $\mathrm{RH}_{s}^{\rho,\theta}$, we get that
\begin{eqnarray*}
\frac{1}{w(Q)}\int_Q \left(e^{\frac{|f(x)|}{\lambda}}-1\right)w(x)\,dx &\leq& \frac{1}{w(Q)}\int_Q e^{\frac{|f(x)|}{\lambda}}w(x)\,dx\\
&\leq&\frac{|Q|}{w(Q)}\left(\frac{1}{|Q|}\int_Q e^{\frac{|f|}{\|f\|_{\tilde{\Phi},Q}}}\right)^{1/s'}\left(\frac{1}{|Q|}\int_Q w^s\right)^{1/s}\\
&\leq&[w]_{\textrm{RH}_s^{\rho, \theta}}\left(1+\frac{r_Q}{\rho(x_Q)}\right)^{\theta}\left(\frac{1}{|Q|}\int_Q e^{\frac{|f|}{\|f\|_{\tilde{\Phi},Q}}}\right)^{1/s'}\\
&=&[w]_{\textrm{RH}_s^{\rho, \theta}}\left(1+\frac{r_Q}{\rho(x_Q)}\right)^{\theta}\left(\frac{1}{|Q|}\int_Q \left(e^{\frac{|f|}{\|f\|_{\tilde{\Phi},Q}}}-1\right)+1 \right)^{1/s'}\\
&\leq& 2^{1/s'}[w]_{\textrm{RH}_s^{\rho, \theta}}\left(1+\frac{r_Q}{\rho(x_Q)}\right)^{\theta},
\end{eqnarray*}
where in the last inequality we have used \eqref{eq: norma realiza el infimo}. We are done.
\end{proof}

\begin{lema}[\cite{BHS11}]\label{lema: control de norma exp por BMO}
Let $\theta>0$, $b\in \mathrm{BMO}_\theta(\rho)$ and $\tilde\Phi(t)=e^t-1$, and $Q=Q(x,r)$. Then there exist positive constants $C$ and $\sigma$ such that 
\[\|b-b_Q\|_{\tilde\Phi, Q}\leq C\|b\|_{\mathrm{BMO}_\theta(\rho)}\left(1+\frac{r}{\rho(x)}\right)^{\sigma}.\]
\end{lema}

The following result establishes an important estimate that will be a key for our purposes. A proof can be found in \cite{BPQ24}.

\begin{lema}\label{lema: lema fundamental}
	Let $u\in A_1^\rho$ and $v\in A_\infty^{\rho}(u)$. Then, there exist $C>0$ and $\theta\geq 0$ such that the following inequality
	\[\frac{uv(Q)}{v(Q)}\leq C\left(1+\frac{r}{\rho(x)}\right)^{\theta}\inf_Q u\]
	holds for every cube $Q=Q(x,r)$.
\end{lema}

The following two propositions show that a stronger smoothness property than \eqref{eq: condicion de suavidad (inf,delta)} or \eqref{suav-horm} can be respectively obtained. A proof for the second one can be found in \cite{BHQ19}.

\begin{propo}\label{propo: suavidad mas fuerte}
    Let $0<\delta\leq 1$ and $K$ be a kernel associated to a SCZO of $(\infty,\delta)$ type. Then for every $0<\tilde\delta<\delta$ and $N>0$ there exists a positive constant $C$ such that the inequality
    \[|K(x,y)-K(x,y_0)|\leq C\frac{|y-y_0|^{\tilde\delta}}{|x-y|^{n+\tilde\delta}}\left(1+\frac{|x-y|}{\rho(x)}\right)^{-N}\]
    holds whenever $|x-y|>2|y-y_0|$.
\end{propo}

\begin{proof}
    Fix $0<\tilde\delta<\delta$ and $N>0$, where $\delta$ is the parameter given by \eqref{eq: condicion de suavidad (inf,delta)}. Choose $0<\sigma<1$ such that $\tilde\delta=\delta(1-\sigma)$. By using \eqref{eq: condicion de tamaño (inf,delta)} with $N_0=N/\sigma$, observe that
    \begin{align*}
      |K(x,y)-K(x,y_0)|^\sigma&\leq |K(x,y)|^\sigma+|K(x,y_0)|^\sigma\\
      &\leq \frac{C_{N_0}^\sigma}{|x-y|^{n\sigma}}\left(1+\frac{|x-y|}{\rho(x)}\right)^{-N_0\sigma}+\frac{C_{N_0}^\sigma}{|x-y_0|^{n\sigma}}\left(1+\frac{|x-y_0|}{\rho(x)}\right)^{-N_0\sigma}\\
      &\leq \frac{C_{N_0}^\sigma}{|x-y|^{n\sigma}}\left(1+\frac{|x-y|}{\rho(x)}\right)^{-N_0\sigma}+\frac{2^{n\sigma}C_{N_0}^\sigma}{|x-y|^{n\sigma}}\left(1+\frac{|x-y|}{2\rho(x)}\right)^{-N_0\sigma}\\
      &\leq (1+2^{(n-N_0)\sigma}) \frac{C_{N_0}^\sigma}{|x-y|^{n\sigma}}\left(1+\frac{|x-y|}{\rho(x)}\right)^{-N},
    \end{align*}
    since $|x-y|>2|y-y_0|$ implies that $2|x-y_0|>|x-y|$. On the other hand, by \eqref{eq: condicion de suavidad (inf,delta)} we get
    \[|K(x,y)-K(x,y_0)|^{1-\sigma}\leq C^{1-\sigma}\frac{|y-y_0|^{\delta(1-\sigma)}}{|x-y|^{(n+\delta)(1-\sigma)}}.\]
    Combining the last two estimates above we obtain
    \[|K(x,y)-K(x,y_0)|\leq C\frac{|y-y_0|^{\tilde\delta}}{|x-y|^{n+\tilde\delta}}\left(1+\frac{|x-y|}{\rho(x)}\right)^{-N}.\qedhere\]
\end{proof}

\begin{propo}\label{propo: suavidad mas fuerte (s,delta)}
    Let $0<\delta\leq 1$, $1<s<\infty$ and $K$ be a kernel associated to a SCZO of $(s,\delta)$ type. Then for every $0<\tilde\delta<\delta$ and $N>0$ there exists a positive constant $C$ such that the inequality
    \[
	\left( \int_{R<|x-y_0|<2R} |K(x,y)-K(x,y_0)|^{s}\,dx\right)^{1/s} 
	\leq C R^{-n/s'} \left(\frac{r}{R}\right)^{\tilde\delta}\left(1+\frac{R}{\rho(y_0)}\right)^{-N}
    \]
    holds for $|y-y_0|<r\leq \rho (y_0)$,  $r<R/2$. 
\end{propo}

\smallskip

The following proposition establishes a useful consequence of the John-Niremberg inequality for the $\text{BMO}_\theta(\rho)$ spaces. A proof can be found in \cite{BHS11}.

\begin{propo}\label{propo: BM0 equivalente a BMOp}
    Let $\theta>0$, $1<s<\infty$ and $b\in \text{BMO}_\theta(\rho)$. Then  the inequality
    \[\left(\frac{1}{|Q|}\int_Q|b(x)-b_Q|^s\right)^{1/s}\leq \|b\|_{\text{BMO}_\theta(\rho)}\left(1+\frac{r}{\rho(x)}\right)^{(N_0+1)\theta}\]
    holds for every cube $Q=Q(x,r)$, where $N_0$ is the constant that appears in \eqref{eq: constantesRho}.
\end{propo}

The following lemma is a extension of the classical result proved, for example, in \cite{Berra-Carena-Pradolini(M)} and we omit it.
\begin{lema}\label{lema: diferencia de promedios acotada por norma BMO}
    Let $b\in\text{BMO}_\theta(\rho)$, $Q=Q(x,r)$ be a cube and $k\in \mathbb{N}$. Then the following estimate holds
    \[|b_Q-b_{2^kQ}|\leq Ck\left(1+\frac{2^kr}{\rho(x)}\right)^{\theta}\|b\|_{\text{BMO}_\theta(\rho)}.\]
\end{lema}

\section{Proof of main results}

We devote this section to the proofs of our main theorems. 

Let us first state a result about weighted strong type inequality for $T_b$, which is an easy consequence on Theorem 8 in \cite{BCH13} and Theorem 4.1 in \cite{VSH20}.

\begin{teo}\label{teo:tipo fuerte del conmutador}
 Let $1<p<\infty$, $\delta\in (0,1]$, $b\in BMO(\rho)$ and $T$ be a SCZO of $(\infty,\delta)$ type. Then there exists a positive constant $C$ such that the inequality
 \[
 \int_{\mathbb{R}^n}\left|T_bf(x)\right|^p\, w(x)\, dx\le C\int_{\mathbb{R}^n}|f(x)|^p \, w(x)\, dx
 \]
holds  for every $w\in A_p^{\rho}$.
\end{teo}
We shall also require the following mixed estimate for SCZO, proved in \cite{BPQ24}.

\begin{teo}\label{teo: mixta para T SCZO infinito, delta}
Let $\rho$ be a critical radius function, $u\in A_1^\rho$ and $v\in A_\infty^\rho(u)$. If $0<\delta\leq 1$ and $T$ is a SCZO of $(\infty,\delta)$ type, then the inequality
	\[uv\left(\left\{x\in\mathbb{R}^n: \frac{|T(fv)(x)|}{v(x)}>\lambda\right\}\right)\leq \frac{C}{\lambda}\int_{\mathbb{R}^n}|f(x)|\,u(x)v(x)dx\]
	holds for every positive $\lambda$.
\end{teo}

\begin{proof}[Proof of Theorem~\ref{teo: mixta para SCZO inf,delta}]
Fix $\lambda>0$. Since $b\in \text{BMO}(\rho)$, there exists $\gamma>0$ such that $b\in \text{BMO}_\gamma(\rho)$. We can assume, without loss of generality, that $\|b\|_{\text{BMO}_\gamma(\rho)}=1$, that is,
\[\left(1+\frac{r}{\rho(x)}\right)^{-\gamma}\frac{1}{|Q|}\int_Q|f-f_Q|\leq 1\]
for every cube $Q=Q(x,r)$.

Since our hypotheses on the weights imply that $v\in A_\infty^{\rho}$, by Remark~\ref{obs: Ap implica duplicante} we can apply Lemma~\ref{lema: CZ decomposition} to perform the Calderón-Zygmund decomposition of $|f|$ at level $\lambda$ with respect to the measure \mbox{$d\mu(x)=v(x)\,dx$ }with $\theta$ sufficiently large to be chosen later. Thus we obtain a sequence of cubes $\{Q_j\}_j$, $Q_j=Q(x_j,r_j)$ such that 
\begin{equation}\label{eq: teo: mixta para SCZO inf,delta - promedios de CZ como lambda}
    \left(1+\frac{r_j}{\rho(x_j)}\right)^{\theta}\lambda \leq \frac{1}{v(Q_j)}\int_{Q_j}|f|\,v\leq C \left(1+\frac{r_j}{\rho(x_j)}\right)^{\sigma}\lambda
\end{equation}
for some $C>0$, $\sigma\geq \theta$ and every $j$. Denoting $J_1=\{j: r_j\leq \rho(x_j)\}$ and $J_2=\mathbb{N}\setminus J_1$, we can define $\Omega_1=\bigcup_{j\in J_1}Q_j$ and $\Omega_2=\bigcup_{j\in J_2} Q_j$. We now write $|f|=g+h_1+h_2$, where
\[g(x)=\left\{\begin{array}{ccl}
    \frac{1}{v(Q_j)}\int_{Q_j}|f|\,v & \text{ if } & x\in Q_j, j\in J_1, \\
     0 & \text{ if } & x\in Q_j, j\in J_2, \\
     |f(x)| & \text{ if } & x\not\in(\Omega_1\cup \Omega_2), \\
\end{array}\right.\]
$h_1$ is given by
\[h_1(x)=\left\{\begin{array}{cl}
    |f(x)|-\frac{1}{v(Q_j)}\int_{Q_j}|f|\,v & \text{ if }  x\in Q_j, j\in J_1, \\
     0  & \text{ otherwise }
\end{array}\right.\]
and $h_2(x)=|f(x)|\mathcal{X}_{\Omega_2}$. If $\Omega=\Omega_1\cup \Omega_2$, $\tilde Q_j=Q(x_j,3\sqrt{n}\,r_j)$ and $\tilde \Omega=\bigcup_j \tilde Q_j$,
 we can proceed as follows
\begin{align*}
    uv\left(\left\{x\in\mathbb{R}^n: \frac{|T_b(fv)(x)|}{v(x)}>\lambda\right\}\right)&\leq  uv\left(\left\{x: \frac{|T_b(gv)(x)|}{v(x)}>\frac{\lambda}{3}\right\}\right)+uv(\tilde\Omega)\\
    &\quad + uv\left(\left\{x\in\mathbb{R}^n\setminus \tilde\Omega: \frac{|T_b(h_1v)(x)|}{v(x)}>\frac{\lambda}{3}\right\}\right)\\
    &\quad + uv\left(\left\{x\in\mathbb{R}^n\setminus \tilde\Omega: \frac{|T_b(h_2v)(x)|}{v(x)}>\frac{\lambda}{3}\right\}\right)\\
    &=A+B+C+D.
\end{align*}
We shall estimate every term above separately. In order to estimate $A$ observe that $v\in A_\infty^\rho(u)$ implies that $v\in A_{q'}^\rho(u)$, for some $q'>1$. Therefore, by Remark~\ref{obs: v en Ap implica v^{1-p'} en Ap'} and item~\eqref{item: lema: propiedades de Ap,rho - item a} in Lemma $\ref{lema: propiedades de Ap,rho}$, we have that $uv^{1-q}\in A_q^\rho$. By using Tchebycheff inequality and Theorem~\ref{teo:tipo fuerte del conmutador} we get
\[ A\leq \frac{3^q}{\lambda^q}\int_{\mathbb{R}^n}|T_b(gv)|^quv^{1-q}
    \leq \frac{C}{\lambda^q}\int_{\mathbb{R}^n}g^quv\leq \frac{C}{\lambda}\int_{\mathbb{R}^n}guv,\]
    because of the fact that $g(x)\leq C\lambda$ since the cubes $Q_j$ for $j\in J_1$ are subcritical and $|f(x)|\le \lambda$ for $x\notin \Omega$.
    
    By means of Lemma~\ref{lema: lema fundamental}, there exist $C>0$ and $\theta_1\geq 0$ such that 
    \[\frac{uv(Q)}{v(Q)}\leq C\left(1+\frac{r}{\rho(x)}\right)^{\theta_1}\inf_Q u\]
    for every cube $Q=Q(x,r)$. Then we can proceed as follows
    \refstepcounter{BPR}\label{pag: estimacion de A}
\begin{align*}
    \int_{\mathbb{R}^n}guv&\leq \sum_{j\in J_1}\frac{uv(Q_j)}{v(Q_j)}\int_{Q_j}|f|\,v+\int_{\mathbb{R}^n\setminus \Omega_1}|f|\,uv\\
    &\leq C\sum_{j\in J_1}\left(1+\frac{r_j}{\rho(x_j)}\right)^{\theta_1}\,\inf_{Q_j}u\int_{Q_j}|f|\,v+\int_{\mathbb{R}^n\setminus \Omega_1}|f|\,uv\\
    &\leq C\int_{\mathbb{R}^n}|f|\,uv,
\end{align*}
since the cubes $Q_j$ for $j\in J_1$ are subcritical. Thus we get the desired estimate for $A$. 

Let us now estimate $B$. Since $uv$ is a doubling weight in the sense of \eqref{eq: propiedad de mu}, there exist $C$ and $\theta_2\geq 0$ such that the inequality
\[uv(2Q)\leq C\,uv(Q)\left(1+\frac{r}{\rho(x)}\right)^{\theta_2}\]
holds for every cube $Q=Q(x,r)$. Combining Lemma~\ref{lema: lema fundamental} with \eqref{eq: teo: mixta para SCZO inf,delta - promedios de CZ como lambda} we proceed as follows
\refstepcounter{BPR}\label{pag: estimacion de B}
\begin{align*}
    B&\le\sum_{j} uv(\tilde Q_j)\\
    &\lesssim \sum_{j}\left(1+\frac{r_j}{\rho(x_j)}\right)^{\theta_2-\theta} \frac{uv( Q_j)}{v(Q_j)}\frac{1}{\lambda}\int_{Q_j}|f|\,v\\
    &\lesssim\sum_{j}\left(1+\frac{r_j}{\rho(x_j)}\right)^{\theta_1+\theta_2-\theta} \frac{\inf_Q u}{\lambda}\int_{Q_j}|f|\, v\\
    &\leq \frac{1}{\lambda}\int_{\mathbb{R}^n}|f|\, uv,
\end{align*}
provided we choose $\theta>\theta_1+\theta_2$. This yields the estimate for $B$.

In order to estimate $C$, recall that 
\[h_1(x)=\sum_{j\in J_1}\left(|f(x)|-\frac{1}{v(Q_j)}\int_{Q_j}|f|v\right)\mathcal{X}_{Q_j}(x)=\sum_{j\in J_1} h_1^j(x),\]
so we can write
\[\frac{T_b(h_1v)}{v}=\sum_{j\in J_1}\frac{T_b(h_1^jv)}{v}=\sum_{j\in J_1}\frac{(b-b_{Q_j})T(h_1^jv)}{v}-\sum_{j\in J_1}\frac{T((b-b_{Q_j})h_1^jv)}{v}.\]
From this estimate we get
\begin{align*}
    C&\leq uv\left(\left\{x\in\mathbb{R}^n\setminus \tilde\Omega: \left|\sum_{j\in J_1}\frac{(b-b_{Q_j})T(h_1^jv)}{v} \right|>\frac{\lambda}{6}\right\}\right)\\
    &\qquad +uv\left(\left\{x\in\mathbb{R}^n\setminus \tilde\Omega: \left|\sum_{j\in J_1}\frac{T((b-b_{Q_j})h_1^jv)}{v} \right|>\frac{\lambda}{6}\right\}\right)\\
    &=C_1+C_2.
\end{align*}
We shall deal with each term above separately. It is clear, from the definition of $h_1^j$, that
\begin{equation}\label{eq: integral 0 de h1j}
    \int_{Q_j}h_1^j(x)v(x)\,dx=0
\end{equation}
for every $j\in J_1$. By applying Tchebycheff inequality together with \eqref{eq: representacion integral de T} and \eqref{eq: integral 0 de h1j} we have
\refstepcounter{BPR}\label{pag: estimacion de C_1}
\begin{align*}
    C_1&\leq \frac{6}{\lambda}\int_{\mathbb{R}^n\setminus \tilde\Omega}\sum_{j\in J_1} |b(x)-b_{Q_j}||T(h_1^jv)(x)|u(x)\,dx\\
    &\leq \frac{C}{\lambda}\int_{\mathbb{R}^n\setminus \tilde\Omega}\sum_{j\in J_1} |b(x)-b_{Q_j}|\left|\int_{Q_j}h_1^j(y)v(y)[K(x,y)-K(x,x_j)]\,dy\right|u(x)\,dx\\
    &\leq \frac{C}{\lambda}\int_{\mathbb{R}^n\setminus \tilde\Omega}\sum_{j\in J_1} |b(x)-b_{Q_j}|\left(\int_{Q_j}|h_1^j(y)[K(x,y)-K(x,x_j)]|v(y)\,dy\right)u(x)\,dx\\
    &\leq \frac{C}{\lambda}\sum_{j\in J_1}\int_{Q_j}|h_1^j(y)|v(y)\left(\int_{\mathbb{R}^n\setminus \tilde Q_j}|b(x)-b_{Q_j}||K(x,y)-K(x,x_j)|u(x)\,dx\right)\,dy\\
    &= \frac{C}{\lambda}\sum_{j\in J_1}\int_{Q_j}|h_1^j(y)|\,v(y)F_j(y)\,dy.
\end{align*}
Observe that $|x-y|>2\,|y-x_j|$ when $x\not\in \tilde Q_j$ and $y\in Q_j$, so if we pick $0<\tilde\delta<\delta$ and $N$ sufficiently large (to be chosen later) by Proposition~\ref{propo: suavidad mas fuerte} we have
\[|K(x,y)-K(x,x_j)|\leq C\left(1+\frac{|x-y|}{\rho(x)}\right)^{-N}\frac{|y-x_j|^{\tilde\delta}}{|x-y|^{n+\tilde\delta}}.\]
%
Since $2Q_j\subseteq \tilde Q_j$, for every $j\in J_1$ we get
\[ F_j(y)\leq C\sum_{k=1}^\infty \int_{2^{k+1}Q_j\setminus 2^kQ_j}|b(x)-b_{Q_j}|\left(1+\frac{|x-y|}{\rho(x)}\right)^{-N}\frac{|y-x_j|^{\tilde\delta}}{|x-y|^{n+\tilde\delta}}\, u(x)\,dx.\]
\begin{afirmacion}\label{af: estimacion de cota del nucleo}
    If $x\in 2^{k+1}Q_j\setminus 2^k Q_j$ and $y\in Q_j$, there exists a positive constant $C$ such that
    \[\left(1+\frac{|x-y|}{\rho(x)}\right)^{-N}\leq C\left(1+\frac{2^{k+1}r_j}{\rho(x_j)}\right)^{-N/(N_0+1)}\]
    where $N_0$ is the constant appearing in \eqref{eq: constantesRho}.
\end{afirmacion}
Applying this claim we arrive to
\begin{align*}
    F_j(y)&\leq C\sum_{k=1}^\infty \int_{2^{k+1}Q_j\setminus 2^kQ_j}|b(x)-b_{Q_j}|\left(1+\frac{|x-y|}{\rho(x)}\right)^{-N}\frac{|y-x_j|^{\tilde\delta}}{|x-y|^{n+\tilde\delta}}\, u(x)\,dx\\
    &\leq C\sum_{k=1}^\infty \frac{\ell(Q_j)^{\tilde\delta}}{\ell(2^kQ_j)^{n+\tilde\delta}}\left(1+\frac{2^{k+1}r_j}{\rho(x_j)}\right)^{-N/(N_0+1)}\int_{2^{k+1}Q_j\setminus 2^kQ_j}|b(x)-b_{Q_j}|\,u(x)\,dx\\
    &\leq C\sum_{k=1}^\infty\left(1+\frac{2^{k+1}r_j}{\rho(x_j)}\right)^{-N/(N_0+1)}\frac{2^{-k\tilde\delta}}{|2^{k+1}Q_j|}\int_{2^{k+1}Q_j}|b(x)-b_{Q_j}|\, u(x)\,dx. 
\end{align*}
In order to deal with the integral above we split the oscillation factor as follows
\begin{align*}
    \frac{1}{|2^{k+1}Q_j|}\int_{2^{k+1}Q_j}|b(x)-b_{Q_j}|\,u(x)\,dx&\leq \frac{1}{|2^{k+1}Q_j|}\int_{2^{k+1}Q_j}|b(x)-b_{2^{k+1}Q_j}|\,u(x)\,dx\\
    &\qquad+\frac{1}{|2^{k+1}Q_j|}\int_{2^{k+1}Q_j}|b_{2^{k+1}Q_j}-b_{Q_j}|\,u(x)\,dx\\
    &=I+II.
\end{align*}
Since $u\in A_1^\rho$, there exists $s>1$ and $\theta_3,\theta_4\geq 0$ such that $u\in \text{RH}_s^{\rho,\theta_3}\cap A_1^{\rho,\theta_4}$.  By applying Hölder inequality together with Proposition~\ref{propo: BM0 equivalente a BMOp} we obtain
\begin{align*}
    I&\leq\left(\frac{1} {|2^{k+1}Q_j|}\int_{2^{k+1}Q_j}|b(x)-b_{2^{k+1}Q_j}|^{s'}\right)^{1/s'}\left(\frac{1}{|2^{k+1}Q_j|}\int_{2^{k+1}Q_j}u^s\right)^{1/s}\\
    &\leq [u]_{\text{RH}_s^{\rho,\theta_3}}\left(1+\frac{2^{k+1}r_j}{\rho(x_j)}\right)^{(N_0+1)\gamma+\theta_3}\left(\frac{1}{|2^{k+1}Q_j|}\int_{2^{k+1}Q_j}u\right)\\
    &\leq [u]_{\text{RH}_s^{\rho,\theta_3}}[u]_{A_1^{\rho,\theta_4}}\left(1+\frac{2^{k+1}r_j}{\rho(x_j)}\right)^{(N_0+1)\gamma+\theta_3+\theta_4}\inf_{2^{k+1}Q_j}u.
\end{align*}
On the other hand, to estimate $II$ we use Lemma~\ref{lema: diferencia de promedios acotada por norma BMO} to get \label{pag: estimacion de I}
\begin{align*}
   II&\leq C(k+1)\left(1+\frac{2^{k+1}r_j}{\rho(x_j)}\right)^{\gamma}\left(\frac{1}{|2^{k+1}Q_j|}\int_{2^{k+1}Q_j}u\right)\\
   &\leq C[u]_{A_1^{\rho,\theta_4}}(k+1)\left(1+\frac{2^{k+1}r_j}{\rho(x_j)}\right)^{\gamma+\theta_4}\inf_{2^{k+1}Q_j}u.
\end{align*}\label{pag: estimacion de II}
Plugging the estimates for $I$ and $II$ in the inequality for $F_j$ we get
\begin{align*}
    F_j(y)&\leq Cu(y)\sum_{k=1}^\infty (k+1)2^{-k\tilde\delta}\left(1+\frac{2^{k+1}r_j}{\rho(x_j)}\right)^{(N_0+1)\gamma+\theta_3+\theta_4-N/(N_0+1)}\\
    &\leq Cu(y),
\end{align*}
for each $y\in Q_j$, provided we choose $N$ sufficiently large such that $(N_0+1)\gamma+\theta_3+\theta_4-N/(N_0+1)<0$. Consequently, using the definition of $h_1$ together with Lemma~\ref{lema: lema fundamental}\refstepcounter{BPR}\label{pag: estimacion de C_1 (2)}
\begin{align*}
  C_1&\leq \frac{C}{\lambda}\sum_{j\in J_1}\int_{Q_j}|h_1^j(y)|\,u(y)v(y)\,dy\\
  &\leq \frac{C}{\lambda}\sum_{j\in J_1}\int_{Q_j}|f(y)|\,u(y)v(y)\,dy+\frac{C}{\lambda}\sum_{j\in J_1}\frac{uv(Q_j)}{v(Q_j)}\int_{Q_j}|f(y)|\,v(y)\,dy\\
  &\leq \frac{C}{\lambda}\int_{\Omega_1}|f(y)|\,u(y)v(y)\,dy+\frac{C}{\lambda}\sum_{j\in J_1}\left(1+\frac{r_j}{\rho(x_j)}\right)^{\theta_1}\inf_{Q_j}u\int_{Q_j}|f(y)|\,v(y)\,dy\\
  &\leq \frac{C}{\lambda}\int_{\mathbb{R}^n}|f(y)\,u(y)v(y)\,dy.
\end{align*}

In order to deal with $C_2$ notice that we can write
\[C_2=uv\left(\left\{x\in\mathbb{R}^n\setminus \tilde\Omega: \frac{\left|T\left(\sum_{j\in J_1}((b-b_{Q_j})h_1^jv\right)\right|}{v} >\frac{\lambda}{6}\right\}\right)\]
since the functions $h_1^j$ are supported on the cubes $Q_j$, which are disjoint. Then we apply Theorem~\ref{teo: mixta para T SCZO infinito, delta} to obtain
\refstepcounter{BPR}\label{pag: estimacion de C_2}
\begin{align*}
    C_2&\leq \frac{C}{\lambda}\int_{\mathbb{R}^n}\left|\sum_{j\in J_1}(b(x)-b_{Q_j})h_1^j(x)\right|\,u(x)v(x)\,dx\\
    &\leq \frac{C}{\lambda}\sum_{j\in J_1}\int_{Q_j}|(b(x)-b_{Q_j})h_1^j(x)|\,u(x)v(x)\,dx\\
    &\leq \frac{C}{\lambda}\sum_{j\in J_1}\int_{Q_j}|(b(x)-b_{Q_j})f(x)|\,u(x)v(x)\,dx+\frac{C}{\lambda}\sum_{j\in J_1}\left(\frac{1}{v(Q_j)}\int_{Q_j}|f|\,v\right)\left(\int_{Q_j}|b-b_{Q_j}|\,uv\right)\\
    &=C_2^1+C_2^2.
\end{align*}

For $C_2^1$ we perform a generalized H\"{o}lder inequality with respect to the measure $d\mu(x)=u(x)v(x)\,dx$ and the functions $\Phi$ and $\tilde\Phi(t)\approx e^{t}-1$ . Then, by Lemma~\ref{lema: control de norma exp por BMO}  and \eqref{eq: norma equivalente} we obtain that
\begin{align*}
    \int_{Q_j}|(b-b_{Q_j})f|uv&\leq uv(Q_j)\,\|b-b_{Q_j}\|_{\tilde\Phi,Q_j,uv}\, \|f\|_{\Phi,Q_j,uv}\\
    &\leq C\,uv(Q_j)\left(1+\frac{r_j}{\rho(x_j)}\right)^{\theta_6 }\inf_{\tau>0}\left\{\tau+\frac{\tau}{uv(Q_j)}\int_{Q_j}\Phi\left(\frac{f}{\tau}\right)uv\right\}\\
    &\leq C\,uv(Q_j)\left(\lambda+\frac{\lambda}{uv(Q_j)}\int_{Q_j}\Phi\left(\frac{f}{\lambda}\right)uv\right)\\
    &\leq C\, uv(Q_j)\lambda +C\lambda\int_{Q_j}\Phi\left(\frac{f}{\lambda}\right)uv.
\end{align*}
for some $\theta_6$ provided by Lemma~\ref{lema: control de norma exp por BMO}. Therefore, by using \eqref{eq: teo: mixta para SCZO inf,delta - promedios de CZ como lambda} and Lemma~\ref{lema: lema fundamental} we get
\begin{align*}
    C_2^1&\leq C\left(\sum_{j\in J_1} uv(Q_j)+\sum_{j\in J_1}\int_{Q_j}\Phi\left(\frac{|f|}{\lambda}\right)uv\right)\\
    &\leq C\sum_{j\in J_1} \frac{uv(Q_j)}{v(Q_j)}\left(1+\frac{r_j}{\rho(x_j)}\right)^{-\theta}\int_{Q_j}\frac{|f|}{\lambda}v+C\int_{\Omega_1}\Phi\left(\frac{|f|}{\lambda}\right)uv\\
    &\leq C\sum_{j\in J_1} \left(1+\frac{r_j}{\rho(x_j)}\right)^{-\theta+\theta_1}\inf_{Q_j}u\int_{Q_j}\frac{|f|}{\lambda}v+C\int_{\Omega_1}\Phi\left(\frac{|f|}{\lambda}\right)uv\\
    &\leq C\int_{\mathbb{R}^n}\Phi\left(\frac{|f|}{\lambda}\right)uv,
\end{align*}
where we have used the fact that the cubes $Q_j$ for $j\in J_1$ are subcritical.

In order to estimate $D$ let us first remember that $h_2(x)=|f(x)|\mathcal{X}_{\Omega_2}$ where $\Omega_2=\bigcup_{j\in J_2} Q_j$.

Then
    $$h_2(x)=\sum_{j\in J_2}|f(x)|\mathcal{X}_{Q_j}=\sum_{j\in J_2}h_2^j(x)$$


 Applying Tchebycheff inequality we have that
\begin{align*}
    D&\leq \frac{3}{\lambda} \int_{\mathbb{R}^n\setminus \tilde \Omega}|T_b(h_2v)(x)|u(x)\,dx\\
    &\leq \frac{3}{\lambda}  \int_{\mathbb{R}^n\setminus \tilde \Omega}\sum_{j\in J_2}\left(\int_{\mathbb{R}^n}|(b(x)-b(y))K(x,y)h_2^j(y)v(y)|\,dy\right) u(x)\,dx\\
    &\leq \frac{3}{\lambda}\sum_{j \in J_2}  \int_{\mathbb{R}^n\setminus \tilde Q_j}\left( \int_{Q_j}|(b(x)-b(y))K(x,y)f(y)v(y)|\,dy\right) u(x)\,dx\\
    &\lesssim \frac{3}{\lambda} \sum_{j\in J_2}\int_{\mathbb{R}^n\setminus \tilde Q_j}|b(x)-b_{Q_j}|\left(\int_{Q_j}|K(x,y)f(y)|v(y)\,dy\right) u(x)\,dx\\
    &\quad +\frac{3}{\lambda}\sum_{j\in J_2} \int_{\mathbb{R}^n\setminus \tilde Q_j}\left(\int_{Q_j}|(b(y)-b_{Q_j})K(x,y)f(y)|v(y)\,dy\right) u(x)\,dx\\
    &=\frac{3}{\lambda}\left(D_1+D_2\right).
\end{align*}
By Tonelli's theorem we can write
\[D_1=\sum_{j \in J_2}  \int_{Q_j}|f(y)|\,v(y)\left(\int_{\mathbb{R}^n\setminus \tilde Q_j}|(b(x)-b_{Q_j})K(x,y)|\,u(x)\,dx\right) \,dy.\]

In order to estimate the inner integral we use \eqref{eq: condicion de tamaño (inf,delta)} with $N>0$ to be chosen, so
\begin{align*}
  \int_{\mathbb{R}^n\setminus \tilde Q_j}|(b(x)-b_{Q_j})K(x,y)|\,u(x)\,dx&\leq \sum_{k=1}^\infty  \int_{2^{k+1}Q_j\setminus 2^k Q_j}\frac{|b(x)-b_{Q_j}|}{|x-y|^n}\left(1+\frac{|x-y|}{\rho(x)}\right)^{-N}u(x)\,dx\\
  &\leq \sum_{k=1}^\infty  \int_{2^{k+1}Q_j\setminus 2^k Q_j}\frac{|b(x)-b_{2^{k+1}Q_j}|}{|x-y|^n}\left(1+\frac{|x-y|}{\rho(x)}\right)^{-N}u(x)\,dx\\
  &\quad +\sum_{k=1}^\infty  \int_{2^{k+1}Q_j\setminus 2^k Q_j}\frac{|b_{2^{k+1}Q_j}-b_{Q_j}|}{|x-y|^n}\left(1+\frac{|x-y|}{\rho(x)}\right)^{-N}u(x)\,dx\\
  &=D_1^1+D_1^2.
\end{align*}

For the estimate of $D_1^1$, we apply Claim~\ref{af: estimacion de cota del nucleo} to obtain
\[D_1^1\leq C\sum_{k=1}^\infty\left(1+\frac{2^{k+1}r_j}{\rho(x_j)}\right)^{-N/(N_0+1)}  \frac{1}{|2^{k+1}Q_j|}\int_{2^{k+1}Q_j\setminus 2^k Q_j}|b(x)-b_{2^{k+1}Q_j}|\,u(x)\,dx.\]
Proceeding as in the estimate of $I$ in page~\pageref{pag: estimacion de I}, we arrive to
\begin{align*}
D_1^1&\leq C\sum_{k=1}^\infty\left(1+\frac{2^{k+1}r_j}{\rho(x_j)}\right)^{(N_0+1)\gamma+\theta_3+\theta_4-N/(N_0+1)}  \inf_{2^{k+1}Q_j}u\\
&\leq C\inf_{Q_j} u\sum_{k=1}^\infty 2^{-k(N/(N_0+1)-\theta_3-\theta_4-(N_0+1)\gamma)}\\
&\leq C\inf_{Q_j}u,
\end{align*}
since $j\in J_2$ and $N$ is such that $N/(N_0+1)-\theta_3-\theta_4-(N_0+1)\gamma>0$.

For $D_1^2$ we use Lemma~\ref{lema: diferencia de promedios acotada por norma BMO} and Claim~\ref{af: estimacion de cota del nucleo} to obtain 
\begin{align*}
  D_1^2&\leq C\sum_{k=1}^\infty(k+1)\left(1+\frac{2^{k+1}r_j}{\rho(x_j)}\right)^{-N/(N_0+1)+\gamma}  \frac{1}{|2^{k+1}Q_j|}\int_{2^{k+1}Q_j}u\\
  &\leq C\sum_{k=1}^\infty k\left(1+\frac{2^{k+1}r_j}{\rho(x_j)}\right)^{-N/(N_0+1)+\gamma+\theta_4}  \inf_{2^{k+1}Q_j}u\\
  &\leq C\inf_{Q_j}u \sum_{k=1}^\infty k2^{-k({N/(N_0+1)-\gamma-\theta_4})}\\
  &\leq C\inf_{Q_j}u,
\end{align*}
since $j\in J_2$ and $N/(N_0+1)-\gamma-\theta_4>0$.

Therefore,
\[D_1\leq C\sum_{j\in J_2}\int_{Q_j}|f|uv\leq C\int_{\mathbb{R}^n}|f|\,uv.\]

For $D_2$ observe that
\begin{align*}
  D_2&=\sum_{j\in J_2} \int_{\mathbb{R}^n\setminus \tilde Q_j}\left(\int_{Q_j}|(b(y)-b_{Q_j})K(x,y)f(y)|\,v(y)\,dy\right) u(x)\,dx\\
  &=\sum_{j\in J_2}\int_{Q_j}|b(y)-b_{Q_j}||f(y)|\left(\int_{\mathbb{R}^n\setminus \tilde Q_j}|K(x,y)|\,u(x)\,dx\right)v(y)\,dy.
\end{align*}

By using \eqref{eq: condicion de tamaño (inf,delta)} and Claim~\ref{af: estimacion de cota del nucleo} we have
\begin{align*}
  \int_{\mathbb{R}^n\setminus \tilde Q_j}|K(x,y)|\,u(x)\,dx&\lesssim \sum_{k=1}^\infty  \int_{2^{k+1}Q_j\setminus 2^k Q_j}|x-y|^{-n}\left(1+\frac{|x-y|}{\rho(x)}\right)^{-N}u(x)\,dx\\
  &\lesssim \sum_{k=1}^\infty  \int_{2^{k+1}Q_j\setminus 2^k Q_j}|x-y|^{-n}\left(1+\frac{2^{k+1}r_j}{\rho(x_j)}\right)^{-N/(N_0+1)}u(x)\,dx\\
  &\lesssim \sum_{k=1}^\infty\left(1+\frac{2^{k+1}r_j}{\rho(x_j)}\right)^{-N/(N_0+1)}\frac{1}{|2^{k+1}Q_j|}  \int_{2^{k+1}Q_j}u(x)\,dx\\
  &\lesssim  \left(1+\frac{r_j}{\rho(x_j)}\right)^{-\tfrac{N}{2(N_0+1)}}\sum_{k=1}^\infty\left(1+\frac{2^{k+1}r_j}{\rho(x_j)}\right)^{-\tfrac{N}{2(N_0+1)}+\theta_4}\inf_{2^{k+1}Q_j}u\\
  &\lesssim \left(1+\frac{r_j}{\rho(x_j)}\right)^{-\tfrac{N}{2(N_0+1)}}\sum_{k=1}^\infty 2^{-k\left(\tfrac{N}{2(N_0+1)}-\theta_4\right)}\, \inf_{Q_j}u\\
  &\lesssim \left(1+\frac{r_j}{\rho(x_j)}\right)^{-\tfrac{N}{2(N_0+1)}} \inf_{Q_j} u,
\end{align*}
since $j\in J_2$ and $N$ is chosen such that $N>2\theta_4(N_0+1)$. Therefore,
\[D_2\lesssim \sum_{j\in J_2}\left(1+\frac{r_j}{\rho(x_j)}\right)^{-\tfrac{N}{2(N_0+1)}}\int_{Q_j}|b(y)-b_{Q_j}||f(y)|\,u(y)v(y)\,dy.\]

We now apply a generalized Hölder inequality with the Young functions $\Phi$ and $\tilde \Phi$, with respect to the measure $d\mu(x)=u(x)v(x)\,dx$  to get
\begin{equation}\label{eq: teo: mixta para SCZO inf,delta - Holder generalizada en D2}
  \int_{Q_j}|b(y)-b_{Q_j}||f(y)|\,u(y)v(y)\,dy\leq uv(Q_j) \|b-b_{Q_j}\|_{\tilde \Phi, Q_j, uv}\|f\|_{\Phi,Q_j, uv}.  
\end{equation}

\refstepcounter{BPR}\label{pag: estimacion para D2}
Since $uv\in A_\infty^\rho$, there exist $r>1$ and $\theta_5\geq 0$ such that $uv\in\textrm{RH}_r^{\rho,\theta_5}$. By Lemma~\ref{lema: con pesos menor que sin pesos} we obtain that
\[\|b-b_{Q_j}\|_{\tilde \Phi, Q_j, uv}\leq C\left(1+\frac{r_j}{\rho(x_j)}\right)^{\theta_5}\|b-b_{Q_j}\|_{\tilde\Phi, Q_j}.\]
On the other hand, using \eqref{eq: norma equivalente} with $w=uv$, the integral in \eqref{eq: teo: mixta para SCZO inf,delta - Holder generalizada en D2} can be bounded by
\begin{align*}
    \int_{Q_j}|b(y)-b_{Q_j}||f(y)|\,u(y)v(y)\,dy&\leq C\lambda \,uv(Q_j)\left(1+\frac{r_j}{\rho(x_j)}\right)^{\theta_5}\|b-b_{Q_j}\|_{\tilde\Phi, Q_j}\\
    &\quad +C\lambda \left(1+\frac{r_j}{\rho(x_j)}\right)^{\theta_5}\|b-b_{Q_j}\|_{\tilde\Phi, Q_j}\int_{Q_j}\Phi\left(\frac{|f|}{\lambda}\right)uv\\
    &=D_2^1+D_2^2.
\end{align*}

We bound these terms separately. Let us first  estimate  $D_2^1$, by Lemma~\ref{lema: control de norma exp por BMO} there exist positive constants $C$ and $\theta_6$  such that
\begin{align*}
D_2^1&\leq C \lambda\left(1+\frac{r_j}{\rho(x_j)}\right)^{\theta_5+\theta_6}\frac{uv(Q_j)}{v(Q_j)}v(Q_j)\\
&\leq C \lambda\left(1+\frac{r_j}{\rho(x_j)}\right)^{\theta_1+\theta_5+\theta_6}\left(\inf_{Q_j} u\right)v(Q_j)\\
&\leq C \lambda\left(1+\frac{r_j}{\rho(x_j)}\right)^{\theta_1+\theta_5+\theta_6-\theta} \int_{Q_j}\frac{|f|}{\lambda}uv,
\end{align*}
where we have also used Lemma~\ref{lema: lema fundamental} and \eqref{eq: teo: mixta para SCZO inf,delta - promedios de CZ como lambda}. Provided we choose $\theta>\theta_1+\theta_5+\theta_6$ we obtain
\[D_2^1\leq C\lambda\int_{Q_j}\frac{|f|}{\lambda}uv.\]
In order to estimate $D_2^2$ we apply again Lemma~\ref{lema: control de norma exp por BMO} to obtain
\[D_2^2\leq C \lambda \left(1+\frac{r_j}{\rho(x_j)}\right)^{\theta_5+\theta_6}\int_{Q_j}\Phi\left(\frac{|f|}{\lambda}\right)uv.\]
From the estimations above we obtain that
\[
D_2\lesssim \sum_{j\in J_2}\lambda \left(1+\frac{r_j}{\rho(x_j)}\right)^{\theta_5+\theta_6-\tfrac{N}{2(N_0+1)}}\int_{Q_j}\Phi\left(\frac{|f|}{\lambda}\right)uv.
\]
By choosing $N$ large enough we obtain the desired result.  
\end{proof}

\bigskip

\begin{proof}[Proof of Claim~\ref{af: estimacion de cota del nucleo}]
    Since $x\in 2^{k+1}Q_j\setminus 2^kQ_j$ and $y\in Q_j$ we have that
    \[|x-y|\geq \frac{\ell(2^{k}Q_j)-\ell(Q_j)}{2}=\frac{(2^k-1)\ell(Q_j)}{2}\geq \frac{2^{k+1}r_j}{4\sqrt{n}}.\]
    On the other hand, by \eqref{eq: constantesRho} we get
    \begin{align*}
        \frac{1}{\rho(x)}&\geq \frac{1}{C_0\rho(x_j)}\left(1+\frac{|x-x_j|}{\rho(x_j)}\right)^{-N_0/(N_0+1)}\\
        &\geq \frac{1}{C_0\rho(x_j)}\left(1+\frac{2^{k+1}r_j}{\rho(x_j)}\right)^{-N_0/(N_0+1)}.
    \end{align*}
    Combining these two estimates we arrive to
    \begin{align*}
        1+\frac{|x-y|}{\rho(x)}&\geq 1+\frac{1}{4C_0\sqrt{n}}\frac{2^{k+1}r_j}{\rho(x_j)}\left(1+\frac{2^{k+1}r_j}{\rho(x_j)}\right)^{-N_0/(N_0+1)}\\
        &\geq \frac{1}{4C_0\sqrt{n}}\left(1+\frac{2^{k+1}r_j}{\rho(x_j)}\right)^{-N_0/(N_0+1)} \left(1+\frac{2^{k+1}r_j}{\rho(x_j)}\right)\\
        &=\frac{1}{4C_0\sqrt{n}}\left(1+\frac{2^{k+1}r_j}{\rho(x_j)}\right)^{1/(N_0+1)}.
    \end{align*}
    Finally,
    \[\left( 1+\frac{|x-y|}{\rho(x)}\right)^{-N}\leq C\left(1+\frac{2^{k+1}r_j}{\rho(x_j)}\right)^{-N/(N_0+1)},\]
    where $C=(4C_0\sqrt{n})^{N}$.
\end{proof}

Before proceeding with the proof of Theorem~\ref{teo: mixta para SCZO s,delta} we state two theorems that will be useful for this purpose. The first one establishes a strong $(p,p)$ type inequality for $T_b$ when $T$ is an $(s,\delta)-SCZO$, whose  proof is also a consequence of Theorem 8 in \cite{BCH13} and Theorem 4.1 in \cite{VSH20}. The second one is a mixed-weak type estimate for the operator $T$ proved in \cite{BPQ24}.

\begin{teo}\label{teo: tipo fuerte del conmutador (s,delta)}
 Let $1<p<\infty$, $\delta\in (0,1]$, $1<s<\infty$, $b\in BMO(\rho)$ and $T$ be a SCZO of $(s,\delta)$ type. Then there exists a positive constant $C$ such that the inequality
 \[
 \int_{\mathbb{R}^n}\left|T_bf(x)\right|^p\, w(x)\, dx\le C\int_{\mathbb{R}^n}|f(x)|^p \, w(x)\, dx
 \]
holds  for every $w\in A_p^{\rho}$.
\end{teo}

\begin{teo}\label{teo: mixta para T SCZO s, delta}
Let $\rho$ be a critical radius function, $1<s<\infty$ and $0<\delta\leq 1$. Let $T$ be
 a SCZO of $(s,\delta)$ type. If $u$ is a weight verifying $u^{s'}\in A_1^{\rho}$ and $v\in A_\infty^{\rho}(u^\beta)$ for some $\beta>s'$, then the inequality
	\[uv\left(\left\{x\in\mathbb{R}^n: \frac{|T(fv)(x)|}{v(x)}>\lambda\right\}\right)\leq \frac{C}{\lambda}\int_{\mathbb{R}^n}|f(x)|\,u(x)v(x)dx\]
	holds for every positive $\lambda$. 
\end{teo}

\begin{proof}[Proof of Theorem~\ref{teo: mixta para SCZO s,delta}] We shall proceed similarly as in the proof of Theorem~\ref{teo: mixta para SCZO inf,delta}, omitting some similar calculations and notations, focusing on the parts of the proof in which different arguments are required. We fix $\lambda>0$ and assume that $\|b\|_{\text{BMO}_\gamma(\rho)}=1$, that is,
there exists $\gamma>0$ such that the inequality
\[\left(1+\frac{r}{\rho(x)}\right)^{-\gamma}\frac{1}{|Q|}\int_Q|f-f_Q|\leq 1\]
holds for every cube $Q=Q(x,r)$.

Since $u^{s'}\in A_1^\rho$ implies that $u\in A_1^{\rho}$, by combining items~\eqref{item: lema: propiedades de Ap,rho - item b} and~\eqref{item: lema: propiedades de Ap,rho - item c} in Lemma~\ref{lema: propiedades de Ap,rho} we get that  $v\in A_\infty^{\rho}$. Then we can perform the Calderón-Zygmund decomposition of $|f|$ at level $\lambda$ with respect to the measure \mbox{$d\mu(x)=v(x)\,dx$} given by Lemma~\ref{lema: CZ decomposition}, with $\theta$ sufficiently large to be chosen later, obtaining a sequence of cubes $\{Q_j\}_j$, $Q_j=Q(x_j,r_j)$ such that 
\begin{equation}\label{eq: teo: mixta para SCZO s,delta - promedios de CZ como lambda}
    \left(1+\frac{r_j}{\rho(x_j)}\right)^{\theta}\lambda \leq \frac{1}{v(Q_j)}\int_{Q_j}|f|v\leq C \left(1+\frac{r_j}{\rho(x_j)}\right)^{\sigma}\lambda
\end{equation}
for some $C>0$, $\sigma\geq \theta$ and every $j$. Performing the same decomposition as in Theorem~\ref{teo: mixta para SCZO inf,delta}, we can write
\begin{align*}
    uv\left(\left\{x\in\mathbb{R}^n: \frac{|T_b(fv)(x)|}{v(x)}>\lambda\right\}\right)&\leq  uv\left(\left\{x: \frac{|T_b(gv)(x)|}{v(x)}>\frac{\lambda}{3}\right\}\right)+uv(\tilde\Omega)\\
    &\quad + uv\left(\left\{x\in\mathbb{R}^n\setminus \tilde\Omega: \frac{|T_b(h_1v)(x)|}{v(x)}>\frac{\lambda}{3}\right\}\right)\\
    &\quad + uv\left(\left\{x\in\mathbb{R}^n\setminus \tilde\Omega: \frac{|T_b(h_2v)(x)|}{v(x)}>\frac{\lambda}{3}\right\}\right)\\
    &=A+B+C+D.
\end{align*}

In order to estimate $A$, our hypotheses on the weights $u$ and $v$ allow us to conclude that $u\in A_1^{\rho}$ and $v\in A_{q'}^\rho(u)$, for some $q'>1$. By Remark~\ref{obs: v en Ap implica v^{1-p'} en Ap'}  and item~\eqref{item: lema: propiedades de Ap,rho - item a} of Lemma~\ref{lema: propiedades de Ap,rho}, we have that $uv^{1-q}\in A_q^\rho$. From Tchebycheff inequality and Theorem~\ref{teo: tipo fuerte del conmutador (s,delta)} we get
\[ A\leq \frac{3^q}{\lambda^q}\int_{\mathbb{R}^n}|T_b(gv)|^quv^{1-q}
    \leq \frac{C}{\lambda^q}\int_{\mathbb{R}^n}g^quv\leq \frac{C}{\lambda}\int_{\mathbb{R}^n}guv,\]
    because of the fact that $g(x)\leq C\lambda$ since the cubes $Q_j$ for $j\in J_1$ are subcritical and $|f(x)|\le \lambda$ for $x\notin \Omega$. Now the desired estimate for $A$ can be achieved by following the same argument as in page~\pageref{pag: estimacion de A}. 

    The estimate of $B$ follows similarly as in page~\pageref{pag: estimacion de B}, whenever we pick $\theta>\theta_1+\theta_2$,  where $\theta_1$ is provided by Lemma~\ref{lema: lema fundamental} and $\theta_2$ is such that the inequality
    \[uv(2Q)\leq Cuv(Q)\left(1+\frac{r}{\rho(x)}\right)^{\theta_2}\]
    holds for every cube $Q=Q(x,r)$.

The estimation of $C$ can be performed as follows
\[\frac{T_b(h_1v)}{v}=\sum_{j\in J_1}\frac{T_b(h_1^jv)}{v}=\sum_{j\in J_1}\frac{(b-b_{Q_j})T(h_1^jv)}{v}-\sum_{j\in J_1}\frac{T((b-b_{Q_j})h_1^jv)}{v},\]
so
\begin{align*}
    C&\leq uv\left(\left\{x\in\mathbb{R}^n\setminus \tilde\Omega: \left|\sum_{j\in J_1}\frac{(b-b_{Q_j})T(h_1^jv)}{v} \right|>\frac{\lambda}{6}\right\}\right)\\
    &\qquad +uv\left(\left\{x\in\mathbb{R}^n\setminus \tilde\Omega: \left|\sum_{j\in J_1}\frac{T((b-b_{Q_j})h_1^jv)}{v} \right|>\frac{\lambda}{6}\right\}\right)\\
    &=C_1+C_2.
\end{align*}
Proceeding as in page~\pageref{pag: estimacion de C_1} we obtain
\begin{align*}
  C_1&\leq \frac{C}{\lambda}\sum_{j\in J_1}\int_{Q_j}|h_1^j(y)|\,v(y)\left(\int_{\mathbb{R}^n\setminus \tilde Q_j}|b(x)-b_{Q_j}||K(x,y)-K(x,x_j)|\,u(x)\,dx\right)\,dy\\
  &\leq \frac{C}{\lambda}\sum_{j\in J_1}\int_{Q_j}|h_1^j(y)|\,v(y)F_j(y)\,dy.
\end{align*}
Fix $j\in J_1$, $0<\tilde\delta<\delta$ and $N>0$ to be chosen later. Since $B_j=B(x_j, 2r_j)\subseteq \tilde Q_j$, by Hölder inequality and Proposition~\ref{propo: suavidad mas fuerte (s,delta)} we have that 
\begin{align*}
    F_j(y)&\leq \sum_{k=1}^\infty \left(\int_{2^{k+1}B_j\setminus 2^kB_j}|b(x)-b_{Q_j}|^{s'}u^{s'}(x)\,dx\right)^{1/s'}\left(\int_{2^{k+1}B_j\setminus 2^kB_j}|K(x,y)-K(x,x_j)|^{s}\,dx\right)^{1/s}\\
    &\leq C\sum_{k=1}^\infty \left(\int_{2^{k+2}Q_j}|b(x)-b_{Q_j}|^{s'}u^{s'}(x)\,dx\right)^{1/s'}(2^{k+1}r_j)^{-n/s'}2^{-(k+1)\tilde\delta}\left(1+\frac{2^{k+1}r_j}{\rho(x_j)}\right)^{-N}.
\end{align*}
Notice that
\begin{align*}
  \int_{2^{k+2}Q_j}|b(x)-b_{Q_j}|^{s'}u^{s'}(x)\,dx&\leq 2^{s'}\left(\int_{2^{k+2}Q_j}|b-b_{2^{k+2}Q_j}|^{s'}u^{s'}+\int_{2^{k+2}Q_j}|b_{2^{k+2}Q_j}-b_{Q_j}|^{s'}u^{s'}\right)\\
  &=I+II.  
\end{align*}
\refstepcounter{BPR}\label{pag: estimacion de I y II}
For $I$, notice that $u^{s'}\in A_1^{\rho}$ implies that there exist $t>1$ and $\theta_3,\theta_4\geq 0$ such that $u^{s'}\in A_1^{\rho,\theta_4}\cap \text{RH}_t^{\rho,\theta_3}$. By applying again Hölder inequality and using Proposition~\ref{propo: BM0 equivalente a BMOp} we get
\begin{align*}
  I&\leq C|2^k Q_j|\left(\frac{1}{|2^{k+2}Q_j|}\int_{2^{k+2}Q_j}|b-b_{2^{k+2}Q_j}|^{s't'}\right)^{1/t'}\left(\frac{1}{|2^{k+2}Q_j|}\int_{2^{k+2}Q_j}u^{s't}\right)^{1/t} \\
  &\leq C\left[u^{s'}\right]_{\text{RH}_t^{\rho,\theta_3}}|2^k Q_j|\left(1+\frac{2^{k+2}r_j}{\rho(x_j)}\right)^{(N_0+1)\gamma s'+\theta_3}\left(\frac{1}{|2^{k+2}Q_j|}\int_{2^{k+2}Q_j}u^{s'}\right)\\
&\leq C\left[u^{s'}\right]_{\text{RH}_t^{\rho,\theta_3}}\left[u^{s'}\right]_{A_1^{\rho,\theta_4}}|2^k Q_j|\left(1+\frac{2^{k+2}r_j}{\rho(x_j)}\right)^{(N_0+1)\gamma s'+\theta_3+\theta_4}\inf_{2^{k+2}Q_j}u^{s'}\\
&=C (2^kr_j)^n\left(1+\frac{2^{k}r_j}{\rho(x_j)}\right)^{(N_0+1)\gamma s'+\theta_3+\theta_4}u^{s'}(y),
\end{align*}
for every $y\in Q_j$. 

To estimate $II$ we apply Lemma~\ref{lema: diferencia de promedios acotada por norma BMO} in order to get
\begin{align*}
    II&\leq C(k+2)\left(1+\frac{2^{k+2}r_j}{\rho(x_j)}\right)^\gamma \int_{2^{k+2}Q_j}u^{s'}\\
    &\leq C\left[u^{s'}\right]_{A_1^{\rho,\theta_4}}(2^k r_j)^n k\left(1+\frac{2^{k+2}r_j}{\rho(x_j)}\right)^{\gamma+\theta_4}\inf_{2^{k+1}Q_j}u^{s'}\\
    &\leq C k (2^kr_j)^n\left(1+\frac{2^{k}r_j}{\rho(x_j)}\right)^{\gamma+\theta_4}u^{s'}(y).
\end{align*}
Therefore, plugging the bounds for $I$ and $II$ in the estimate for $F_j(y)$ we get
\begin{align*}
    F_j(y)&\leq Cu(y)\sum_{k=1}^\infty (2^kr_j)^{n/s'}\left(1+\frac{2^{k}r_j}{\rho(x_j)}\right)^{(N_0+1)\gamma+(\theta_3+\theta_4)/s'-N}k^{1/s'} (2^kr_j)^{-n/s'}2^{-k\tilde\delta}\\
    &\leq Cu(y),
\end{align*}
provided we choose $N$ sufficiently large such that $N>(N_0+1)\gamma+(\theta_3+\theta_4)/s'$. Consequently, we obtain that
\[C_1\leq \frac{C}{\lambda}\sum_{j\in J_1}\int_{Q_j}|h_1^j(y)|\,u(y)v(y)\,dy.\]
From now on, the desired estimate is obtained as in page~\pageref{pag: estimacion de C_1 (2)}. 

To take care of $C_2$, we observe that
\[C_2=uv\left(\left\{x\in\mathbb{R}^n\setminus \tilde\Omega: \left|\frac{T(\sum_{j\in J_1}(b-b_{Q_j})h_1^jv)}{v} \right|>\frac{\lambda}{6}\right\}\right).\]
By Theorem~\ref{teo: mixta para T SCZO s, delta} we obtain
\[C_2\leq \frac{C}{\lambda}\int_{\mathbb{R}^n}\left|\sum_{j\in J_1}(b(x)-b_{Q_j})\,h_1^j(x)\right|u(x)v(x)\,dx,\]
and the estimate can be continued exactly as in page~\pageref{pag: estimacion de C_2}.

Finally, applying Tchebycheff inequality for $D$ yields
\begin{align*}
    D&\leq \frac{3}{\lambda} \sum_{j\in J_2}\int_{\mathbb{R}^n\setminus \tilde Q_j}|b(x)-b_{Q_j}|\left(\int_{Q_j}|K(x,y)f(y)|\,v(y)\,dy\right) u(x)\,dx\\
    &\quad +\frac{3}{\lambda}\sum_{j\in J_2} \int_{\mathbb{R}^n\setminus \tilde Q_j}\left(\int_{Q_j}|(b(y)-b_{Q_j})K(x,y)f(y)|\,v(y)\,dy\right) u(x)\,dx\\
    &=\frac{3}{\lambda}(D_1+D_2).
\end{align*}
In order to estimate $D_1$ we interchange the order of integration to get
\[D_1\leq\sum_{j\in J_2} \int_{Q_j}|f(y)|\,v(y)\left(\int_{\mathbb{R}^n\setminus \tilde Q_j}|b(x)-b_{Q_j}||K(x,y)|\,u(x)\,dx\right)\,dy.\]
Denoting again $B_j=B(x_j,2r_j)$, by Hölder inequality together with the size condition~\eqref{TamHorm} we obtain
\begin{align*}
    \int_{\mathbb{R}^n\setminus \tilde Q_j}|b(x)-b_{Q_j}||K(x,y)|\,u(x)\,dx&\leq \sum_{k=1}^\infty \left(\int_{2^{k+1}B_j\setminus 2^kB_j}|K(x,y)|^s\,dx\right)^{1/s}\\
    &\qquad \times \left(\int_{2^{k+1}B_j\setminus 2^kB_j} |b(x)-b_{Q_j}|^{s'}u^{s'}(x)\,dx\right)^{1/s'}\\
    &\leq C_N\sum_{k=1}^\infty (2^{k+1}r_j)^{-n/s'}\left(1+\frac{2^{k+1}r_j}{\rho(x_j)}\right)^{-N}\left(\int_{2^{k+2}Q_j}|b-b_{Q_j}|^{s'}u^{s'}\right)^{1/s'}
\end{align*}
where $N>0$ will be chosen later. By following the same estimate for $I$ and $II$ in page~\pageref{pag: estimacion de I y II} we deduce that
\[\left(\int_{2^{k+2}Q_j}|b-b_{Q_j}|^{s'}u^{s'}\right)^{1/s'}\leq C (2^kr_j)^{n/s'}\left(1+\frac{2^{k}r_j}{\rho(x_j)}\right)^{(N_0+1)\gamma+(\theta_3+\theta_4)/s'}k^{1/s'}u(y),\]
for every $y\in Q_j$. So
\begin{align*}
  \int_{\mathbb{R}^n\setminus \tilde Q_j}|b(x)-b_{Q_j}||K(x,y)|\,u(x)\,dx&\leq C_N u(y)\sum_{k=1}^\infty 2^{-k(N-(N_0+1)\gamma-(\theta_3+\theta_4)/s')}k^{1/s'}\\
  &\leq C u(y),
\end{align*}
since $r_j>\rho(x_j)$ for $j\in J_2$ and
$N$ verifies $N>(N_0+1)\gamma+(\theta_3+\theta_4)/s'$. Therefore,
\[D_1\leq C\sum_{j\in J_2}\int_{Q_j}|f(y)|\,u(y)v(y)\,dy\leq C\int_{\mathbb{R}^n}|f(y)|\,u(y)v(y)\,dy.\]
For $D_2$ we proceed similarly, that is
\[D_2\leq \int_{Q_j}|b(y)-b_{Q_j}||f(y)|\,v(y)\left(\int_{\mathbb{R}^n\setminus \tilde Q_j}|K(x,y)|\,u(x)\,dx\right)\,dy.\]
Performing Hölder inequality and using the size condition~\eqref{TamHorm} we obtain
\begin{align*}
    \int_{\mathbb{R}^n\setminus \tilde Q_j}|K(x,y)|\,u(x)\,dx&\leq \sum_{k=1}^\infty \left(\int_{2^{k+1}B_j\setminus 2^kB_j}|K(x,y)|^s\,dx\right)^{1/s} \left(\int_{2^{k+1}B_j\setminus 2^kB_j} u^{s'}(x)\,dx\right)^{1/s'}\\
    &\leq C_N\sum_{k=1}^\infty (2^{k+1}r_j)^{-n/s'}\left(1+\frac{2^{k+1}r_j}{\rho(x_j)}\right)^{-N}\left(\int_{2^{k+2}Q_j}u^{s'}\right)^{1/s'}\\
    &\leq C_N\left[u^{s'}\right]_{A_1^{\rho,\theta_4}}^{1/s'}\sum_{k=1}^\infty \left(1+\frac{2^{k+1}r_j}{\rho(x_j)}\right)^{-N+\theta_4/s'}u(y)\\
    &\leq C\, u(y)\left(1+\frac{r_j}{\rho(x_j}\right)^{-N/2}\sum_{k=1}^\infty 2^{-k(N/2-\theta_4/s')}\\
    &\leq C\, u(y)\left(1+\frac{r_j}{\rho(x_j}\right)^{-N/2},
\end{align*}
whenever $N>2\theta_4/s'$. This allows us to conclude that
\[D_2\leq C\sum_{j\in J_2}\left(1+\frac{r_j}{\rho(x_j)}\right)^{-N/2}\int_{Q_j}|b(y)-b_{Q_j}||f(y)|\,u(y)v(y).\]
From this estimate and proceeding as we did in page~\pageref{pag: estimacion para D2}, we get the desired inequality provided $N$ is chosen sufficiently large.
\end{proof}

\def\cprime{$'$}
\providecommand{\bysame}{\leavevmode\hbox to3em{\hrulefill}\thinspace}
\providecommand{\MR}{\relax\ifhmode\unskip\space\fi MR }
\providecommand{\MRhref}[2]{%
  \href{http://www.ams.org/mathscinet-getitem?mr=#1}{#2}
}
\providecommand{\href}[2]{#2}

\end{document}